\documentclass[12pt]{amsart} 
\usepackage{amssymb, colordvi} 
\usepackage{multirow} 
\usepackage{graphicx} 
 
\usepackage{ulem} 
 
\numberwithin{equation}{section} 
\newtheorem{theorem}{Theorem} 
\newtheorem{proposition}[theorem]{Proposition} 
 
\newtheorem{cor}[theorem]{Corollary}

\theoremstyle{definition} 
\newtheorem{definition}[theorem]{Definition} 
 
\theoremstyle{definition} 
 
\newtheorem{remark}[theorem]{Remark} 
 
\newcounter{FNC}[page] 
\def\fauxfootnote#1{{\addtocounter{FNC}{2}$^\fnsymbol{FNC}$%
     \let\thefootnote\relax\footnotetext{$^\fnsymbol{FNC}$#1}}} 
 
\newcommand{\C}{{\mathbb{C}}} 
 
\renewcommand{\P}{{\mathbb{P}}}

\newcommand{\co}{{\;:\;}} 
 
\newcommand{\be}{{\bf e}}

\newcommand{\diag}{{\mbox{\rm diag}}} 
\newcommand{\tr}{{\mbox{\rm Tr}}}

\newcommand{\DeCo}[1]{\Blue{#1}} 
\newcommand{\demph}[1]{\DeCo{{\sl #1}}} 
 
\title{Complex static skew-symmetric output feedback control} 
 
\author{Christopher J. Hillar} 
\address{Christopher J. Hillar\\The Mathematical Sciences Research Institute\\ 
         17 Gauss Way\\ 
         Berkeley, CA 94720-5070}
\email{chillar@msri.org} 
\urladdr{http://www.msri.org/people/members/chillar/} 
\author{Frank Sottile} 
\address{Frank Sottile \\ Department of Mathematics\\ 
         Texas A\&M University\\ 
         College Station\\ 
         Texas \ 77843\\ 
         USA} 
\email{sottile@math.tamu.edu} 
\urladdr{http://www.math.tamu.edu/\~{}sottile/} 
\subjclass{93B55 (14M15 93B27)} 
 
\thanks{Hillar supported in part by an NSF Postdoctoral Fellowship and an NSA 
  Young Investigator grant}  
\thanks{Sottile supported in part by the Institute for Mathematics and its Applications,  
    Institut Mittag-Leffler, and NSF grants DMS-0701050 and DMS-1001615.}   
\keywords{pole placement, feedback control, orthogonal Grassmannian, Lagrangian 
  Grassmannian, skew-symmetric matrix}  
\begin{document} 
 
\begin{abstract} 
 We study the problem of feedback control for skew-symmetric and skew-Hamiltonian transfer 
 functions using skew-symmetric controllers. 
 This extends work of Helmke, et al., who studied static symmetric 
 feedback control of symmetric and Hamiltonian linear systems. 
 We identify spaces of linear systems with symmetry as natural subvarieties 
 of the moduli space of rational curves in a Grassmannian, 
 give necessary and sufficient conditions for pole placement by static  
 skew-symmetric complex feedback, and use Schubert calculus for the orthogonal  
 Grassmannian to count the number of complex feedback laws when there are finitely many
 of them. 
 Finally, we also construct a real skew-symmetric linear system 
 with only real feedback for any set of real poles.  
\end{abstract} 
\maketitle 
 
%
%
\section{Introduction} 
Many fundamental questions about output feedback pole assignment for general linear systems 
have been answered by appealing to algebraic geometry, and more specifically 
to the geometry of Grassmann manifolds. 
This body of work has led to important contributions in systems theory: 
Hermann and Martin gave necessary and sufficient conditions for complex static output  
feedback control~\cite{HM1, HM2, HM3}, Brockett and Byrnes used Schubert calculus to count 
the number of pole-assigning feedback laws~\cite{BB81}, and then Rosenthal~\cite{Ro94} and 
Ravi, Rosenthal, and Wang~\cite{RRW1,RRW2} solved these problems for complex dynamic 
compensators using quantum Schubert calculus. 
For a description of the earlier literature, we recommend~\cite{By89}. 
This line of work on complex feedback has led to a solution of the problem of pole-assignment in the 
real case, some of which is found in~\cite{EG,RSW,W92}. 
Likewise, it has influenced work in algebraic geometry~\cite{HV,So00b,SS}, some of 
which is surveyed in~\cite{So01}. 
 
The \demph{Lagrangian Grassmannian} and \demph{orthogonal Grassmannian} are subsets of the 
usual Grassmannian, and in principle they should also appear in systems theory. 
This was realized by Helmke, Rosenthal, and 
Wang~\cite{HRW06}, who studied the control of linear systems with symmetric and Hamiltonian  
state-space realizations by static symmetric output feedback. 
They gave necessary and sufficient conditions for pole placement by static symmetric complex  
feedback, linking this problem to the Schubert calculus on Lagrangian Grassmannians, and then used  
this link to count the number of complex feedback laws  
when there are finitely many of them.  
 
We continue this line of research. 
We first identify spaces of linear systems  
with a natural symmetry as certain  
subvarieties of the space of rational curves in a Grassmannian. 
More specifically, we consider linear systems with McMillan degree $n$  
whose transfer function $G(s)$ (a square matrix of rational functions which is 
defined at $s=\infty$) has one of the following four symmetries: 
 \begin{enumerate} 
  \item $G(s)^T=G(s)$  \quad symmetric, 
  \item $G(s)^T=G(-s)$  \quad Hamiltonian  ($n$ must be even),  
  \item $G(s)^T=-G(-s)$ \quad skew-Hamiltonian, and  
  \item $G(s)^T=-G(s)$  \quad skew-symmetric  ($n$ must be even). 
 \end{enumerate} 
Symmetries (1)-(2) were studied in~\cite{HRW06} and occur naturally in 
systems theory~\cite{BD,F83}. 
Stabilization of symmetric systems (1) by real symmetric output feedback was 
considered in~\cite{MaHe}, where it was shown that there may be no real feedback 
laws placing $n$ real poles when $n\geq m$. 
The symmetries (3)-(4) are natural to consider from the point of view of algebraic 
geometry.   
Theorem~\ref{Th:three} gives an example of a real $m$-input $m$-output skew-symmetric 
linear system of McMillan degree $2\binom{m}{2}$ such that every feedback law is real when 
placing real poles, demonstrating that it is feasible to place real
poles with real skew-symmetric feedback.  
 
Let $A$ be a nondegenerate bilinear form on $\C^{2m}$. 
The annihilator \DeCo{$H^{\perp_A}$} of a plane $H$ in $\C^{2m}$ is  
     the set of $v\in\C^{2m}$ such that $A(v,w)=0$ for all $w\in H$. 
The annihilator of an $m$-plane in $\C^{2m}$ is  
also an $m$-plane, and so the association $H\mapsto H^{\perp_A}$ defines an involution \DeCo{$\iota_A$} on the Grassmannian 
\demph{$G(m,2m)$}.  
If $A$ is skew-symmetric, then the set of fixed points of $\iota_A$ is the  
\demph{Lagrangian Grassmannian, $LG(m)$}, which is a manifold of dimension 
$\binom{m{+}1}{2}$. 
If $A$ is symmetric, then the set of fixed points of $\iota_A$ has two isomorphic 
components, either of which forms the $\binom{m}{2}$-dimensional  
\demph{orthogonal Grassmannian, $OG(m)$}, also called the \demph{spinor variety}~\cite{FuPr}. 
 
It is classical (e.g., proved by Gaussian elimination \cite{BW66}) that any two invertible complex symmetric matrices $A,B$ are (transpose) \textit{congruent}:  there exists an invertible matrix $X$ such that $X^{\top}AX = B$. 
Similarly, any two invertible complex skew-symmetric matrices are congruent.  
Thus, we will always assume that our forms are $\langle x,y\rangle= x^{\top}Ay$, where $A$ is either 
 \begin{equation}\label{Eq:forms} 
   \DeCo{O_{2m}}\ :=\ \left[\begin{matrix}0&I_m\\ I_m&0\end{matrix}\right] 
   \qquad\mbox{or}\qquad 
   \DeCo{J_{2m}}\ :=\ \left[\begin{matrix}0&I_m\\-I_m&0\end{matrix}\right]\,, 
 \end{equation} 
where $I_m$ is the $m\times m$ identity matrix. 
We omit the subscripts $m$ and $2m$ when the dimensions are clear from context. 
A general linear subspace $H\in G(m,2m)$ is the row space of a matrix of the form 
$[I_m \co F]$, where $F$ is an $m\times m$ matrix. 
A calculation shows that $\iota_J(H)$ is spanned by $[I_m\co F^T]$ and  
$\iota_O(H)$ is spanned by $[I_m\co -F^T]$, so that $H\in LG(m)$  
if and only if $F$ is symmetric and $H\in OG(m)$ if and only if $F$  
is skew-symmetric. 
 
We reach the same conclusion for $H$ of the form $[F\co I_m]$. 
Such isotropic planes form dense open subsets of $LG(m)$ and $OG(m)$. 
 
If we associate an $m$-input $m$-output proper transfer function $G(s)$ of McMillan 
degree $n$ to the row space of the matrix 
\[ 
   [I_m\ :\ G(s)]\,, 
\] 
we obtain a map $\gamma\colon \P^1\to G(m,2m)$ of degree $n$, where $\P^1$ is the 
  complex projective line. 
The image is the \demph{Hermann-Martin curve}~\cite{HM3} of $G(s)$, which we will identify 
  with $G(s)$. 
The set of all such proper transfer functions forms a dense open subset in the space 
of rational curves of degree $n$ in the Grassmannian $G(m,2m)$~\cite{Ro94}. 
Our first main result identifies sets of transfer functions with symmetries as natural 
subvarieties of the space of rational curves in the Grassmannian. 
 
\begin{theorem}\label{Th:one} 
 The following hold over the ground field $\C$. 

$(1)$  The set of symmetric  linear systems with $m$ inputs and $m$ outputs of McMillan degree 
  $n$ is an irreducible quasiprojective manifold of dimension $(m{+}1)n + \tbinom{m{+}1}{2}$. 
  It is a dense open  
  subset of the space of rational curves of degree $n$ in $LG(m)$.  
 
$(2)$ The set of Hamiltonian linear systems with $m$ inputs and $m$ outputs of even McMillan 
  degree $n$ is an irreducible quasiprojective  manifold of dimension $mn + \tbinom{m{+}1}{2}$. 
  It is a  subset 
  of the space of rational curves $\gamma$ in the Grassmannian $G(m,2m)$ that satisfy: 
 \begin{equation}\label{Eq:Hamil_curves} 
    \gamma(-s)\ =\ \iota_J(\gamma(s))\,. 
 \end{equation} 
 
$(3)$ The set of skew-Hamiltonian linear systems with $m$ inputs and $m$ outputs of McMillan 
  degree $n$ is an irreducible quasiprojective  manifold of dimension $mn + \tbinom{m}{2}$. 
  It is a subset 
  of the space of rational curves $\gamma$ in the Grassmannian $G(m,2m)$ that satisfy: 
 \begin{equation}\label{Eq:Sk_Hamil_curves} 
    \gamma(-s)\ =\ \iota_O(\gamma(s))\,. 
 \end{equation} 
 
$(4)$ The set of skew-symmetric linear systems with $m$ inputs and $m$ outputs of even 
  McMillan degree $n=2\ell$ is an irreducible quasiprojective  manifold of dimension $(m{-}1)n + \tbinom{m}{2}$. 
  It is naturally a dense open 
  subset of the space of rational curves of degree $\ell$ in 
  $OG(m)$.  
\end{theorem}  
 
The proof of Theorem~\ref{Th:one} is straightforward and given in Section~\ref{S:one}, 
following a proof of a version of the Kalman Realization 
Theorem~\cite[Theorem~6.2-4]{Ka80} for symmetric transfer functions. 
We do not know if Hamiltonian systems form dense open subsets of the space of curves 
satisfying~\eqref{Eq:Hamil_curves}, or if skew-Hamiltonian systems form dense open subsets 
of the space of curves satisfying~\eqref{Eq:Sk_Hamil_curves}, for these spaces of curves have 
yet to be studied.

As sets of linear systems with symmetry are identified with irreducible quasiprojective algebraic 
varieties, the notion of genericity for complex systems makes sense. 
That is, a property is \demph{generic} if it holds on a nonempty Zariski open subset 
(which is therefore dense) of the corresponding space. 
 
Because symmetric and skew-symmetric transfer functions are open subsets of the moduli 
spaces of rational curves in the Lagrangian Grassmannian and orthogonal Grassmannian, 
respectively, output feedback control by either static or dynamic symmetric and 
skew-symmetric linear systems is related to Schubert calculus on these Grassmannians,  
both classical (for static feedback laws) and quantum (for dynamic feedback). 
The main result of~\cite{HRW06} concerned static symmetric feedback. 
We establish the analogous result for static skew-symmetric feedback. 
 
The symmetry of skew-Hamiltonian and skew-symmetric linear systems is preserved by 
static skew-symmetric output feedback, so it is natural to place poles with 
static skew-symmetric controllers. 
The poles of a skew-Hamiltonian transfer function are invariant under multiplication by 
$-1$, so there are essentially only $\lfloor n/2\rfloor$ poles to place. 
Here, $\lfloor x \rfloor$ is the greatest integer not exceeding the real number $x$.
Similarly, poles of a skew-symmetric linear system occur with even 
multiplicity, and therefore a skew-symmetric linear system of even McMillan degree $n$  
has only $n/2$ poles to place. 
Our second main theorem gives necessary and sufficient conditions for pole placement with  
static skew-symmetric feedback. 
 
\begin{theorem}\label{Th:two} 
  A generic strictly proper skew-symmetric (respectively skew-Hamiltonian) transfer 
  function $G(s)$ with $m$ inputs, $m$ outputs, and McMillan degree $n$ is pole-assignable 
  with complex static skew-symmetric feedback compensators if and only if  
   $\lfloor n/2\rfloor \leq \binom{m}{2}$. 
\end{theorem} 
 
In particular, when $m=4$, we see that a skew-symmetric system of McMillan degree  $12$ or 
less is pole-assignable with complex static skew-symmetric feedback compensators, and 
skew-Hamiltonian systems with $m=4$ and McMillan degree $13$ or less are pole-assignable. 
We remark that our proof does not determine the dense open subset of pole-assignable
symmetric  linear systems.  
 
Our third main result counts the number of feedback laws for a generic skew-symmetric system 
of McMillan degree $2\binom{m}{2}$. 
It also shows that there exist systems with real feedback 
laws, in a strong way---for these systems, every feedback law placing real poles is real. 
 
\begin{theorem}\label{Th:three} 
  A generic skew-symmetric linear system of McMillan degree  
  $2\binom{m}{2}$ has exactly 
 \[ 
   \DeCo{d_m}\ :=\  
   \binom{m}{2}!  \frac{1!\dotsb(m-2)!}{1! 3!\dotsb(2m-3)!} 
 \] 
 static complex skew-symmetric 
 controllers that place a given general set of $\binom{m}{2}$ poles. 
 
Moreover, for every $m$, there exists a real skew-symmetric linear system with $m$ inputs and outputs and 
 McMillan degree $2\binom{m}{2}$ such that for every choice of $\binom{m}{2}$ real poles, 
 there are are $d_m$ feedback laws, and every one is real. 
\end{theorem}  
 
We prove these theorems in Section~\ref{S:two}. 
The argument for Theorem~\ref{Th:two} is influenced by the proof in~\cite{HRW06}, 
but it is a considerable simplification. 
The skew-symmetric system with $d_m$ real feedback laws  
comes from the Wronski map in the Schubert calculus, and the result on reality is a 
restatement of a theorem of Purbhoo~\cite{Purbhoo}. 
 
We do not address questions about dynamic feedback. 
If we use a dynamic compensator of McMillan degree $q$ to place the poles of a 
linear system of McMillan degree $n$ with one of these symmetries, then a calculation shows 
that the resulting system (of McMillan degree $n+q$) has the same symmetry as the original 
system if and only if the compensator had that same symmetry.  
Thus it is natural to consider dynamic control when both the system and compensator have the 
same symmetry. 
A dimension count gives the necessary condition that $n+q$ be at most 
\[ 
   (m{+}1)q+\binom{m+1}{2}\,,\quad 
    mq+\binom{m+1}{2}\,,\quad 
    mq+\binom{m}{2}\,,\quad\mbox{and}\quad 
   (m{-}1)q+\binom{m}{2}\,, 
\] 
for generic pole placement of symmetric, Hamiltonian, skew-Hamiltonian, and skew-symmetric 
linear systems by dynamic controllers of the same symmetry.  
We do not know if these conditions are sufficient---this requires the generic surjectivity of 
the corresponding pole-placement map. 
 
If this dimension condition is necessary and sufficient, then the quantum Schubert calculus 
for $LG(m)$ and $OG(m)$~\cite{KT03,KT04} may be used to count the number of dynamic 
compensators for symmetric or skew-symmetric systems (also~\cite{Ruffo} for symmetric 
compensators).  
Counting dynamic compensators for 
Hamiltonian and skew-Hamiltonian systems requires a deeper 
study of the corresponding spaces of curves, for it will involve orbifold quantum cohomology~\cite{CR02}.   
 
Our results and analysis also apply to discrete-time linear systems with 
these symmetries, in the same way that continuous-time transfer functions are related 
to discrete-time transfer functions when there are no symmetries.

We thank Joachim Rosenthal who explained his work to us and encouraged us to extend his 
results to skew-symmetric transfer functions. We also thank Uwe Helmke for discussions on the 
systems theory background.

%
%
\section{Geometry of state-space realization with symmetries}\label{S:one} 
 
We study the state-space realizations of transfer functions with symmetry and  
identify the spaces of such transfer functions as  
certain subvarieties of the moduli spaces of rational curves in Grassmannians. 
Portions of this material are classical or can be found in~\cite{HRW06}, but we include some 
proofs for completeness. We work entirely over the complex numbers. 
Write \Blue{$X^{\top}$} for the transpose of a matrix $X$ and \Blue{$X^{-\top}$} for 
$(X^{-1})^{\top}=(X^{\top})^{-1}$.  
Recall that a square matrix $X$ is \demph{symmetric} if $X^{\top}=X$ and  
\demph{skew-symmetric} if $X^{\top}=-X$. 
 
Let $J$ be the $2\ell\times 2\ell$ matrix, 
\[ 
   \DeCo{J}\ :=\ \left[\begin{matrix}0&I\\-I&0\end{matrix}\right], 
\] 
where $I$ is the $\ell\times\ell$ identity matrix. 
Note that $J^{\top}=-J=J^{-1}$. 
A $2\ell\times 2\ell$ matrix $X$ is \demph{Hamiltonian} if $XJ$ is symmetric and 
\demph{skew-Hamiltonian} if $XJ$ is skew-symmetric.

Let $m$ and $n$ be positive integers, which we assume are fixed throughout. 
We write $\ell$ for $\lfloor n/2\rfloor$. 
Suppose that we have a time-dependent complex linear system with inputs $u\in\C^m$,  
outputs $y\in\C^m$, and McMillan degree $n$. 
This has a minimal state-space realization: 
 \begin{equation}\label{Eq:State-Space} 
  \begin{array}{rcl} 
    \dot{x}&=& Ax + Bu\\ 
          y&=& Cx + Du\,, 
   \end{array} 
 \end{equation} 
where $A\in\C^{n\times n}$, $B\in\C^{n\times m}$, $C\in\C^{m\times n}$, and  
$D\in\C^{m\times m}$, and corresponding (proper) transfer function, 
\[ 
  \DeCo{G(s)}\ :=\ C(sI-A)^{-1}B + D\,. 
\] 
The \demph{poles} of the transfer function $G(s)$ are the eigenvalues of $A$. 
The transfer function $G(s)$ is \demph{strictly proper} if $D=0$. 
 
Given a strictly proper linear system, a static linear feedback law is given by an 
$m\times m$ matrix $F$, where we set $u=Fy+v$. 
The resulting linear system is  
 \begin{equation}\label{Eq:controlled_system} 
   \dot{x}\ =\ (A+BFC)x + Bv\qquad y\ =\ Cx\,, 
 \end{equation} 
and its transfer function has poles at the roots of 
 \[
    \varphi(s)\ :=\ \det(sI-(A+BFC))\,.
 \] 
A fundamental problem is:  When is it possible to choose $F$ to obtain a given choice of 
monic polynomial $\varphi(s)$? 
A system is \demph{pole-assignable} if the map $F\mapsto \varphi(s)$ is dominant (its image  
is a dense subset of the set of polynomials $\varphi(s)$). 
Our main result concerns the pole-assignability of generic linear systems with 
skew-Hamiltonian and skew-symmetric symmetry. 
We first study the spaces of linear systems with symmetry.\smallskip

The complex general linear group $GL(n)$ of invertible $n\times n$ matrices acts  
on the space of realizations~\eqref{Eq:State-Space} via 
\begin{equation*}\label{GLn_action} 
  X.(A,B,C,D)\ \longmapsto\ (X^{-1}AX,\, X^{-1}B,\, CX, D)\,, 
\end{equation*} 
where $X\in GL(n)$, and it preserves the transfer function. 

If we restrict this action to the dense open set of  
minimal state space realizations, 
then the Kalman Realization Theorem~\cite[Theorem~6.2-4]{Ka80} identifies the orbits with 
transfer functions and shows that $GL(n)$ acts without fixed points. 
In particular, if $(A,B,C,D)$ and $(\alpha,\beta,\gamma,\delta)$ are both minimal state 
space realizations of the same transfer function, then there is a unique $X\in GL(n)$ such 
that  
\[ 
   X.(A,B,C,D)\ =\ (\alpha,\beta,\gamma,\delta)\,. 
\] 
We first extend these classical facts to transfer 
functions with symmetries.  
 
\begin{definition} 
 A transfer function $G(s)$ is  
 \demph{symmetric},  
 \demph{Hamiltonian},  
 \demph{skew-Hamiltonian},  or 
 \demph{skew-symmetric} if for all $s\in \C$ we have, 
 \[ 
    G(s)^{\top}\ =\ G(s)\,,\quad 
    G(s)^{\top}\ =\ G(-s)\,,\quad 
    G(s)^{\top}\ =\ -G(-s)\,,\quad\mbox{or}\quad 
    G(s)^{\top}\ =\ -G(s)\,, 
 \] 
 respectively. 
\end{definition} 
 
State-space realizations may also have symmetries. 
 
\begin{definition}\label{Def:state-space} 
 A realization~\eqref{Eq:State-Space} is \demph{symmetric} if 
 $A$ is symmetric, $B=C^{\top}$, and $D$ is symmetric. 
 Symmetric realizations have symmetric transfer functions: 
 \[ 
   G(s)^{\top}\ =\ B^{\top}(sI-A^{\top})^{-1}C^{\top} + D^{\top} 
         \ =\ C(sI-A)^{-1}B + D\ =\  G(s)\,. 
 \] 
 
 A realization~\eqref{Eq:State-Space} with $n$ even is \demph{Hamiltonian}  
 if $A$ is Hamiltonian, $JB=C^{\top}$, and $D$ is symmetric. 
 Note that $A^{\top}=JAJ$ and $B^{\top}=CJ$. 
 Hamiltonian realizations have Hamiltonian transfer functions: 
 \begin{eqnarray*} 
   G(-s)^{\top}\ =\  
   B^{\top}(-sI-A^{\top})^{-1}C^{\top}+D^{\top} &=& CJ(sJIJ-JAJ)^{-1}JB + D^{\top}\\ 
                           &=& C(sI-A)^{-1}B + D^{\top}\ \ =\ G(s)\,. 
 \end{eqnarray*} 

 A realization~\eqref{Eq:State-Space} is \demph{skew-Hamiltonian} if 
 $A$ is skew-symmetric, $B=C^{\top}$, and $D$ is skew-symmetric. 
 Skew-Hamiltonian realizations have skew-Hamiltonian transfer functions: 
 \begin{eqnarray*} 
   -G(-s)^{\top}&=& -B^{\top}(-sI-A^{\top})^{-1}C^{\top} - D^{\top}\\ 
           &=&-C(-sI+A)^{-1}B + D\ =\  G(s)\,. 
 \end{eqnarray*} 
 
 Finally, a realization~\eqref{Eq:State-Space} with $n$ even is  
 \demph{skew-symmetric} if $A$ is skew-Hamiltonian, $JB=C^{\top}$, and $D$ is skew-symmetric. 
 In this case, $A^{\top}=-JAJ$ and $B^{\top}=CJ$. 
 Skew-symmetric realizations have skew-symmetric transfer functions: 
 \begin{eqnarray*} 
   -G(s)^{\top}&=&-B^{\top}(sI-A^{\top})^{-1}C^{\top} - D^{\top}\\ 
          &=&-CJ(-sJIJ+JAJ)^{-1}(-JB) +D\ =\ G(s)\,. 
 \end{eqnarray*} 
\end{definition} 
 
\begin{remark} 
 In symmetric and Hamiltonian realizations, the matrix $A$ has the same type 
 (symmetric or Hamiltonian, respectively), while for skew-Hamiltonian and skew-symmetric 
 realizations, the matrix $A$ has the opposite type; it is skew-symmetric or 
 skew-Hamiltonian, respectively.  
\end{remark} 
 
While there is no {\it a priori} reason that a transfer function 
with one of these symmetries
would have a minimal state-space realization with the
  same symmetry,  
that is indeed the case. 
The \demph{orthogonal group $O(n)$} is the subgroup of $GL(n)$ consisting of complex 
matrices $X$ with $X^{\top}X=I$, and when $n$ is even, the  \demph{symplectic group $Sp(n)$} 
is the subgroup of $GL(n)$ consisting of complex matrices $X$ with $X^{\top}JX=J$.  
We establish the analog of the Kalman Realization 
Theorem for transfer functions with symmetry.  
 
\begin{proposition}\label{prop:realizations} 
 A transfer function has one of the symmetry types---symmetric, Hamiltonian, 
 skew-Hamiltonian, or skew-symmetric---if and only if it has a complex minimal state-space 
 realization having the corresponding symmetry type. 
 
 Furthermore, if  $(A,B,C,D)$ and $(\alpha,\beta,\gamma,\delta)$ are two such minimal 
 state-space  
 realizations of the same transfer function, then there is a unique matrix 
 $X\in O(n)$ (respectively $X\in Sp(n)$) such that  
 $X.(A,B,C,D)=(\alpha,\beta,\gamma,\delta)$ for symmetric and 
 skew-Hamiltonian transfer  
 functions, (respectively for Hamiltonian and skew-symmetric transfer functions). 
\end{proposition} 
 
Following the proof for symmetric transfer functions in~\cite{HRW06} (see also~\cite{FH95}), 
we give the proof in the cases of skew-Hamiltonian and skew-symmetric transfer functions. 
The case of Hamiltonian transfer functions is similar, and is due to Brockett and 
Rahimi~\cite{BR72}.  Also, the first half, concerning symmetric realizations, is due to Brockett~\cite{Br78}. 
 
\begin{proof} 
 We prove the forward implication in the first statement as we have already shown that a 
 state-space realization having one of these symmetries gives a transfer function with the 
 same symmetry. 
 
Suppose that $G(s)=-G(-s)^{\top}$ is a skew-Hamiltonian transfer function with minimal 
 state-space realization $(A,B,C,D)$.  
 Since 
 \[  
   -G(-s)^{\top}\ =\  -B^{\top}(-sI-A^{\top})^{-1} C^{\top}-D^{\top}\ =\ B^{\top}(sI+A^{\top})^{-1}C^{\top}-D^{\top}\,, 
 \] 
 $(-A^{\top}, C^{\top}, B^{\top}, -D^{\top})$ is also a minimal realization. 
 By the Kalman Realization Theorem, there is a unique invertible 
 matrix $X$ such that  
 \begin{eqnarray*} 
   (A,\ B,\ C,\ D)&=& X . (-A^{\top},\ C^{\top},\ B^{\top},\ -D^{\top}) \\ 
   &=& (-X^{-1}A^{\top}X,\ X^{-1}C^{\top},\ B^{\top}X,\ -D^{\top})\\ 
              &=& (-X^{-1}(-X^{-1}A^{\top}X)^{\top}X,\ X^{-1}(B^{\top}X)^{\top},\ (X^{-1}C^{\top})^{\top}X,\ -D^{\top})\\ 
              &=& (X^{-1}X^{\top}AX^{-\top}X,\ X^{-1}X^{\top}B,\ CX^{-\top}X,\ -D^{\top}). \\ 
          &=&(X^{-1}X^{\top}). (A,\ B,\ C,\ D). 
 \end{eqnarray*} 
Since $(A,\ B,\ C,\ D) = I. (A,\ B,\ C,\ D)$, another application of the Kalman
Realization Theorem gives us that $X^{-1}X^{\top}=I$; thus, $X$ is symmetric.  
   
An invertible complex symmetric matrix $X$ admits a Tagaki factorization $X=Y^{\top}Y$, 
with $Y$ invertible \cite[Corollary 4.4.4]{HJ1}.  
 Then the realization $(YAY^{-1}, YB, CY^{-1})$ of the 
 transfer function $G(s)$ is skew-Hamiltonian. 
 Indeed, 
 \begin{eqnarray*} 
   (YAY^{-1})^{\top}& =& Y^{-\top}A^{\top}Y^{\top}\ =\ -Y^{-\top}X^{\top}AX^{-\top}Y^{\top}\\ 
   & =& -Y^{-\top}Y^{\top}YAY^{-1}Y^{-\top}Y^{\top}\ =\ -YAY^{-1}\,. 
 \end{eqnarray*} 
 Similarly, $(YB)^{\top}=CY^{-1}$. 
 Indeed, $C=B^{\top}X$ so $B^{\top}=CX^{-1}$, and thus 
\[ 
   (YB)^{\top}\ =\ B^{\top}Y^{\top}\ =\ CX^{-1} Y^{\top}\ =\ C(Y^{-1}Y^{-\top})Y^{\top} 
   \ =\ CY^{-1}\,. 
\] 
 
 Now suppose that $(A,B,C,D)$ and $(\alpha,\beta,\gamma,\delta)$ are minimal 
 skew-Hamiltonian realizations of the same transfer function. 
 Let $X\in GL(n)$ be the unique matrix such that  
\[ 
   X.(A,B,C,D)\ =\  (X^{-1}AX,X^{-1}B,CX,D)\ =\  
   (\alpha,\beta,\gamma,\delta)\,. 
\] 
 Then we have  
\[ 
  \alpha\ =\ -\alpha^{\top}\ =\ -X^{\top}A^{\top}X^{-\top}\ =\ X^{\top}AX^{-\top}\,, 
\] 
 and similarly $\beta=X^{\top}B$ and $\gamma=CX^{-\top}$ so that  
 $X^{-\top}.(A,B,C,D)=(\alpha,\beta,\gamma,\delta)$. 
 It follows that $X^{-\top}=X$, by the uniqueness of $X$. 
 But then $X^{\top}X=I$ and so $X\in O(n)$ is orthogonal.\medskip

Consider next the case that $G(s)=-G(s)^{\top}$ is a skew-symmetric transfer function with
minimal state-space realization $(A,B,C,D)$.  
 Since $J^{-1}=J^{\top}$, we see that $-G(s)^{\top}$ equals 
 \begin{eqnarray*}  
  -B^{\top}(sI-A^{\top})^{-1} C^{\top} -D^{\top} &=& B^{\top}J^{\top}J(-sI+A^{\top})^{-1}JJ^{\top}C^{\top}-D^{\top} \\ 
             &=&  (JB)^{\top}(-sJ^{-1}IJ^{-1}+J^{\top}A^{\top}J^{\top})^{-1}(CJ)^{\top}-D^{\top} \\ 
             &=&  (JB)^{\top}(sI-(-JAJ)^{\top})^{-1}(CJ)^{\top}-D^{\top} \,, 
 \end{eqnarray*} 
 and so $(-(JAJ)^{\top}, (CJ)^{\top}, (JB)^{\top}, -D^{\top})$ is also a minimal realization of $G(s)$. 
 Then there is a unique $X\in GL(n)$ such that  
 \begin{eqnarray*}  
   (A,\ B,\ C,\ D)&=& X.(-(JAJ)^{\top},\ (CJ)^{\top},\ (JB)^{\top},\ -D^{\top})\\ 
              &=& (-X^{-1}(JAJ)^{\top}X,\ X^{-1}(CJ)^{\top},\ (JB)^{\top}X,\ -D^{\top})\,. 
 \end{eqnarray*} 
 Substituting the equality into itself and simplifying, we see that  
\begin{equation} \label{Eq:SIB} 
(A,B,C, D) =  (X^{-1}J^{\top}X^{\top}JAJX^{-\top}J^{\top}X,\ X^{-1}J^{\top}X^{\top}JB,\  CJ X^{-\top}J^{\top}X,\ D). 
 \end{equation} 
The right-hand side of~\eqref{Eq:SIB} is not immediately seen to have the form  
 $R.(A,B,C,D)$, 
 but if we set $R:=JX^{-\top}J^{\top}X$ and use that $-J^{\top}=J^{-\top}$ and $J^{-1}=-J$,  
 we obtain 
\[ 
  R^{-1}\ =\ X^{-1}J^{-\top}X^{\top}J^{-1}\ =\ X^{-1}J^{\top}X^{\top}J\,, 
\] 
 which shows that~\eqref{Eq:SIB} is $(R^{-1}AR, R^{-1}B, CR, D)$.  
 We conclude that $R=I$, so that $-I=JX^{-\top}J^{\top}X=J^{-\top}X^{-\top}JX$, and so  
 $(JX)^{\top}=-JX$ is skew-symmetric. 
 
Paralleling the argument for skew-Hamiltonian transfer functions, we will use this to 
obtain a skew-symmetric realization. 
The key ingredient is a factorization of the matrix $JX$. 
 An invertible complex skew-symmetric matrix $Z$ admits a Tagaki-like factorization $Z=Y^{\top}JY$. 
 Indeed, $Z = USU^{\top}$ for a unitary matrix $U$ and a block diagonal $S$ which is a direct sum of  
$2\times 2$ blocks of the form 
$\left[\begin{smallmatrix}0&a\;\\-a&0\end{smallmatrix}\right]$ with 
$a \in \mathbb C \setminus\{0\}$ (e.g., see \cite[Problem 26 in Chapter 4.4]{HJ1}). 
 Thus, after block scaling with blocks 
 $\left[\begin{smallmatrix}\sqrt{a}&0\\0&\sqrt{a}\end{smallmatrix}\right]$ and 
applying a permutation similarity, we arrive at the claimed factorization. 
 
 In the factorization $JX=Y^{\top}JY$ of the invertible skew-symmetric matrix $JX$,  
 the matrix $Y$ is also invertible. 
 Then the realization $(YAY^{-1}, YB, CY^{-1}, D)$ of the 
 transfer function $G(s)$ is skew-symmetric. 
 Indeed, as $D=-D^{\top}$ and we have $J=-J^{\top}=J^{-\top}$,  $X^{\top}J^{\top}=-JX=-Y^{\top}JY$, and  
 $J^{-\top}X^{-\top}=Y^{-1}JY^{-\top}$, we obtain 
 \begin{eqnarray*} 
   (YAY^{-1}J)^{\top}& =& J^{\top}Y^{-\top}A^{\top}Y^{\top}\ =\  
         -J^{\top}Y^{-\top} (X^{-1}(JAJ)^{\top}X)^{\top} Y^{\top}\\ 
   & =& -J^{\top}Y^{-\top}X^{\top}JAJX^{-\top} Y^{\top} \ =\ 
          J^{\top}Y^{-\top}X^{\top}J^{\top}AJ^{-\top}X^{-\top} Y^{\top}\\ 
   & =& -J^{\top}Y^{-\top}Y^{\top}JY  AY^{-1}JY^{-\top} Y^{\top}\ =\ -Y  AY^{-1}J\,. 
 \end{eqnarray*} 
 Similarly, $CY^{-1}=(JB)^{\top}XY^{-1}=(J YB)^{\top}$, 
 so the realization $(YAY^{-1}, YB, CY^{-1}, D)$ of $G(s)$ is skew-symmetric. 
\smallskip

 Finally, suppose that $(A,B,C,D)$ and $(\alpha,\beta,\gamma,\delta)$ are minimal 
 skew-symmetric realizations of the same transfer function. 
 Let $X\in GL(n)$ be the unique matrix such that  
\[ 
   X.(A,B,C,D)\ =\  (X^{-1}AX,X^{-1}B,CX,D)\ =\  
   (\alpha,\beta,\gamma,\delta)\,. 
\] 
 Recall that $-A^{\top}=JAJ$ and $\alpha =-J\alpha^{\top}J$, and so  
\[ 
  \alpha \ =\ -J\alpha^{\top}J\ =\ -JX^{\top}A^{\top}X^{-\top}J 
   \ =\ JX^{\top}JAJX^{-\top}J \ =\ J^{-1}X^{\top}JAJ^{-1}X^{-\top}J\,. 
\] 
 Similarly recall that $C^{\top}=JB$ and $\beta=J\gamma^{\top}$, so that  
\[ 
  \beta\ =\ J^{\top}\gamma^{\top}\ =\ J^{\top}(CX)^{\top}\ =\ J^{\top}X^{\top}C^{\top}\ =\ -JX^{\top}JB\ =\ J^{-1}X^{\top}JB\,. 
\] 
 Since we also have $\gamma=CJ^{-1}X^{-\top}J$, we see that  
 $(\alpha,\beta,\gamma,\delta)=R.(A,B,C,D)$, where $R=J^{-1}X^{-\top}J$. 
  By the uniqueness of $X$, we have $X= J^{-1}X^{-\top}J$ so that  
 $X^{\top} JX=J$ and so $X\in Sp(n)$ is symplectic. 
\end{proof} 
 
We use this proposition to compute the dimensions of the corresponding spaces of transfer 
functions/rational curves, the first part of the proof of Theorem~\ref{Th:one}.  
 
\begin{cor}\label{C:dim} 
 The set of transfer functions with a fixed symmetry is an irreducible 
 quasiprojective complex algebraic manifold. 
 For symmetric, Hamiltonian, skew-Hamilton\-ian, and skew-symmetric transfer functions, these have respective 
 dimensions  
\[ 
   (m{+}1)n\ +\ \tbinom{m{+}1}{2}\,,\quad 
         mn\ +\ \tbinom{m{+}1}{2}\,,\quad 
         mn\ +\ \tbinom{m}{2}\,,\quad\mbox{and}\quad 
   (m{-}1)n\ +\ \tbinom{m}{2}\,. 
\] 
\end{cor} 
 
\begin{proof} 
 The space of minimal complex symmetric  realizations is a Zariski-open subset of an affine space. 
 By Proposition~\ref{prop:realizations}, the set of 
 transfer functions with a fixed symmetry type is identified with the set of orbits of 
 an algebraic group ($O(n)$ or $Sp(n)$) acting freely on the space of minimal symmetric 
 realizations, which is an open subset of a vector space.   
 The first statement of the corollary follows from this as the set of such orbits has a 
 natural structure as an irreducible smooth complex algebraic variety~\cite[Th.~9.16]{Lee}.   
 
 For the second, we note that the dimension of the orbit space is the 
 difference of the dimensions of the space of symmetric realizations and of 
 the group. 
 The orthogonal group $O(n)$ has dimension $\binom{n}{2}$ and the  
 symplectic group $Sp(n)$ has dimension $\binom{n{+}1}{2}$~\cite{FuH}. 
 Spaces of symmetric and Hamiltonian realizations both have the same dimension 
\[ 
   \binom{n{+}1}{2}\ +\ nm\ +\ \binom{m{+}1}{2}\,, 
\] 
 while the spaces of skew-Hamiltonian and skew-symmetric realizations both have dimension  
\[ 
  \binom{n}{2}\ +\ nm\ +\ \binom{m}{2}\,. 
\] 
 The corollary now follows. 
\end{proof} 
 
We now complete the proof of Theorem~\ref{Th:one} by identifying the Hermann-Martin 
  curves of symmetric and skew-symmetric transfer functions with dense open subsets of the 
appropriate spaces of rational curves in the Lagrangian and Orthogonal Grassmannians. 
First note that if $G(s)$ is symmetric (respectively skew-symmetric) then for $s\in\P^1$, the row space $K(s)$ of the matrix  
$[I_m\co G(s)]$ lies in $LG(m)$ (respectively in $OG(m)$). 
Thus the Hermann-Martin curve is a curve in $LG(m)$ (respectively in $OG(m)$). 
 
To finish, we show that that these Hermann-Martin curves are  
dense in the corresponding spaces of rational curves, which is a consequence of their 
having the same dimension. 
The space of rational curves in $LG(m)$ of degree $d$ has dimension  
$d(m{+}1)+\binom{m{+}1}{2}$ and the space of rational curves in $OG(m)$ of degree $d$ 
has dimension $2d(m{-}1)+\binom{m}{2}$~\cite{KT03,KT04}. 
Thus Theorem~\ref{Th:one} follows if we knew that a curve in $LG(m)$ of McMillan degree 
$n$ has degree $n$ in $LG(m)$ and a curve in $OG(m)$ of McMillan degree $n=2\ell$ has 
degree $\ell$ in $OG(m)$. 
 
These facts are well-known.
The McMillan degree of a curve is it degree in the classical
  Grassmannian $G(m,2m)$ in its Pl\"ucker embedding.
The inclusion $LG(m)\hookrightarrow G(m,2m)$ arises from a linear map on the standard 
projective embedding of $LG(m)$, so rational curves of degree $n$ in $LG(m)$ have degree 
$n$ in $G(m,2m)$, and hence McMillan degree $n$.  
On the other hand, the inclusion $OG(m)\hookrightarrow G(m,2m)$ arises from 
the second Veronese map on the natural projective embedding of $OG(m)$. 
%
Thus a rational curve of degree $\ell$ in $OG(m)$ will have degree $2\ell$ in $G(m,2m)$, 
and hence McMillan degree $2\ell$.
This completes the proof of Theorem~\ref{Th:one}.

%
%
\section{Static skew-symmetric state feedback control}\label{S:two} 
 
Suppose that we have a strictly proper ($D = 0$) skew-Hamiltonian or skew-symmetric linear system 
with a minimal state-space realization 
\begin{equation}\label{Eq:system} 
   \dot{x}\ =\ Ax+Bu\qquad\quad y\ =\ Cx\,. 
\end{equation} 
If we introduce a static linear state-space feedback law $u=Fy+v$, then the new system 
 \begin{equation}\label{Eq:feedback} 
   \dot{x}\ =\ (A+BFC)x+Bv\qquad\quad y\ =\ Cx\, 
 \end{equation} 
has the same symmetry as the original system when $F$ is skew-symmetric. 
(The elementary calculation is given below.) 
We investigate the control of such linear systems with complex skew-symmetric static 
state-space feedback. 
We first establish the necessary and sufficient conditions for generic pole placement of 
Theorem~\ref{Th:two}, relate skew-symmetric feedback control to the Schubert calculus 
on the orthogonal Grassmannian, and then prove Theorem~\ref{Th:three}, counting the number 
of feedback laws that place a generic set of poles of a generic skew-symmetric transfer 
function with McMillan degree $2\binom{m}{2}$.  
We do not yet know how to count the controllers of a skew-Hamiltonian transfer function of 
McMillan degree $n$ when $\binom{m}{2}=\lfloor n/2\rfloor$. 
 
\begin{proof}[Proof of Theorem~$\ref{Th:two}$] 
We give the proof for generic pole-assignability of skew-symmetric transfer 
functions and indicate how the argument changes for skew-Hamiltonian transfer functions. 
This follows and simplifies the arguments in~\cite{HRW06}. 
 
We identify skew-symmetric $N\times N$ matrices with the vector space $\wedge^2\C^N$, 
where the elementary decomposable tensor $e_i\wedge e_j$ ($i \neq j$) corresponds to the matrix having 
1 in position $(i,j)$, $-1$ in position $(j,i)$, and 0 in other positions. 
 
Suppose that~\eqref{Eq:system} is skew-symmetric, so that $A$ is skew-Hamiltonian, 
$(AJ)^{\top}=-AJ$, and $C=B^{\top}J$.  
If we have a feedback law $u=Fy+v$ where $F$ is skew-symmetric, then 
the new system~\eqref{Eq:feedback} remains skew-symmetric as  
$A+BFB^{\top}J$ is skew-Hamiltonian. 
Thus, the characteristic polynomial  
 \begin{equation}\label{Eq:char_ss} 
   \varphi(s)\ :=\ \det(sI - (A + BFB^{\top}J))\ =\ \det(sJ - (AJ - BFB^{\top})), 
 \end{equation} 
is the determinant of a skew-symmetric matrix and is therefore a square (its determinant is the square of its Pfaffian). 
Thus it is natural to ask for skew-symmetric feedback laws $F$ which place these 
$\ell=n/2$ roots (which are poles of the transfer function). 
 
The pole placement map sends a 
skew-symmetric matrix $F\in\wedge^2\C^m$ to the degree $2\ell$ polynomial 
$\varphi(s)$. 
Since this monic polynomial is a square, its last $\ell$ coefficients (those of 
$s^{2\ell-1},\dotsc,s^{\ell}$) determine its first $\ell$ coefficients. 
These coefficients are, up to a sign, the elementary symmetric functions of the 
eigenvalues of $A+BFB^{\top}J$. 
By the Newton identities, these coefficients determine, and are determined by, the Newton 
power sums which are the traces of 
$(A+BFB^{\top}J)^k$ for $k=1,\dotsc,\ell$. 
 
To show generic pole-assignability, we only need to exhibit one choice of matrices $A,B$ for 
which the map  
\[ 
   \DeCo{\Psi}\ \colon\  
    \wedge^2 \C^m\ni F\ \longmapsto\  
    (\tr((A+BFB^{\top}J)^k)\mid k=1,\dotsc,\ell)\ \in\ \C^\ell 
\] 
is dominant. 
We do this by showing that the differential $d\Psi_0$ at $0\in\wedge^2\C^m$ is  
surjective.  
 
Let $\alpha_1,\dotsc,\alpha_\ell$ be distinct numbers and $\beta_1,\dotsc,\beta_m$ be 
numbers such that the $\binom{m}{2}$ products $\beta_i\beta_j$ for $i<j$ are distinct, and 
such that $\beta_i^\ell\neq\beta_j^\ell$, for every $i\neq j$. 
Let $D=\diag(\alpha_1,\dotsc,\alpha_\ell)$ be the diagonal matrix with entries 
$\alpha_1,\dotsc,\alpha_\ell$, and let $A$ be the block diagonal matrix  
$\left[\begin{smallmatrix}D&0\\0&D\end{smallmatrix}\right]$. 
Finally, let $B$ be the matrix with entries $\beta_j^{i-1}$ for $i=1,\dotsc,2\ell$ and 
$j=1,\dotsc,m$.  
For this choice of $A,B$, the differential $d\Psi$ is surjective at 
$0\in\wedge^2\C^m$, for this implies that the image contains an open set in the
  classical topology, and thus a Zariski-open subset.  
 
To see surjectivity, note that the differential at $0$ is the linear map 
 \begin{equation}\label{Eq:dPsi0} 
   d\Psi_0\ \colon F\ \longmapsto\  
  (k\cdot \tr(A^{k-1}\cdot BFB^{\top}J)\mid k=1,\dotsc,\ell)\,. 
 \end{equation} 
Consider this map on the basis element $e_i\wedge e_j$ of $\wedge^2\C^m$. 
A direct calculation shows that $B(e_i\wedge e_j)B^{\top}$ is the vector $b_i\wedge b_j$, 
where $b_1,\dotsc,b_m$ are the columns of $B$. 
For our choice of $B$, the $(p,q)$-entry of $b_i\wedge b_j$ is 
\[ 
   \det\left[\begin{matrix}\beta_i^{p-1}&\beta_j^{p-1}\\ 
                           \beta_i^{q-1}&\beta_j^{q-1}\end{matrix}\right] 
  \ =\ (\beta_i\beta_j)^{p-1}(\beta_j^{q-p}-\beta_i^{q-p})\,. 
\] 
It follows that the map $d\Psi_0$ sends the vector $e_i\wedge e_j$ to the vector 
\[ 
   \Bigl(k\sum_{p=1}^{2\ell} (A^{k-1}(b_i\wedge b_j) J)_{p,p}\;\Big\vert\; 
    k=1,\dotsc,\ell\Bigr)\,. 
\] 
Since
  $A^{k-1}=\diag(\alpha_1^{k-1},\dotsc,\alpha_\ell^{k-1},\alpha_1^{k-1},\dotsc,\alpha_\ell^{k-1})$,
\[
   \bigl((b_i\wedge b_j)J\bigr)_{p,p}\ =\ \left\{
     \begin{array}{rcl} -(b_i\wedge b_j)_{p,p+\ell}&\ &\mbox{if }p\leq\ell\\
                         (b_i\wedge b_j)_{p,p-\ell}&\ &\mbox{if }p>\ell
     \end{array}\right.\ ,
\]
and $(b_i\wedge b_j)$ is skew-symmetric, this vector is  
\[ 
   \Bigl(2k\sum_{p=1}^{\ell}  
     \alpha_p^{k-1}(\beta_i\beta_j)^{p-1}(\beta_i^{\ell}-\beta_j^{\ell})\;\Big\vert\; 
    k=1,\dotsc,\ell\Bigr)\,. 
\] 
Thus, $d\Psi_0$ is represented by the $\ell\times\binom{m}{2}$ matrix which is the 
product of the matrices 
\[ 
   \diag(2,4,\dotsc,2\ell)\cdot 
    \bigl(\alpha_p^{k-1}\bigr)_{p=1,\dotsc,\ell}^{k=1,\dotsc,\ell} \cdot 
    \bigl( (\beta_i\beta_j)^{p-1}\bigr)_{1\leq i<j\leq m}^{p=1,\dotsc,\ell} 
   \cdot \diag( \beta_i^\ell-\beta_j^\ell | 1\leq i<j\leq m)\,, 
\] 
and so its rank is the minimum of $\ell$ and $\binom{m}{2}$, which proves Theorem~\ref{Th:two} 
for skew-symmetric linear systems. 
\medskip 
 
Suppose now that~\eqref{Eq:system} is skew-Hamiltonian, so that $A$ is skew-symmetric  
and $C=B^{\top}$.   
Under a skew-symmetric feedback law $u=Fy+v$,  
the new system~\eqref{Eq:feedback} remains skew-Hamiltonian 
as $A+BFB^{\top}$ is skew-symmetric. 
The characteristic polynomial  
 \begin{equation}\label{Eq:char_sH} 
   \varphi(s)\ :=\ \det(sI - (A + BFB^{\top})) 
 \end{equation} 
satisfies $\varphi(s)=(-1)^n\varphi(-s)$, so its nonzero roots $\lambda$ occur in pairs 
$\pm\lambda$.  
Thus it is natural to ask for skew-symmetric feedback laws $F$ which place these 
$\ell=\lfloor n/2\rfloor$ pairs of roots (which are poles of the transfer function). 
 
The pole placement map sends a 
skew-symmetric matrix $F\in\wedge^2\C^m$ to the degree $n$ polynomial 
$\varphi(s)$. 
As before, we investigate the surjectivity of the pole-placement map by considering the  
map associating Newton power sums, which is 
 \[ 
  \Psi\ \colon\ \wedge^2\C^m\ni F\ \longmapsto 
   \  (\tr(A+BFB^{\top})^k\mid k=1,\dotsc,n)\ \in \C^\ell\,. 
 \] 
Since $A+BFB^{\top}J$ is skew-symmetric, the trace is zero if $k$ is odd, which is why the 
codomain of this map is $\C^\ell$. 
Thus, the differential of $\Psi$ at $0$ is  
\[ 
  d\Psi_0\ \colon\ F\ \longmapsto \ (2k\cdot\tr(A^{2k-1}BFB^{\top})\mid k=1,\dotsc,\ell)\,. 
\] 
Let $D$ be the same diagonal matrix as before. 
If $n$ is even, let $A$ be the block matrix  
$\left[\begin{smallmatrix}0&D\\-D&0\end{smallmatrix}\right]$ and $B$ be the same as 
before.  
If $n$ is odd, then add a new first row and column of $0$s to $A$ and extend $B$ with the row 
$(\beta_i^{2\ell} \, | \, i=1,\dotsc,m)$. 
Then nearly the same calculation as before shows that $d\Psi_0$ is surjective when 
$\ell\leq\binom{m}{2}$, which completes the proof of Theorem~\ref{Th:two}. 
\end{proof} 

Before proving Theorem~\ref{Th:three}, we first recall some standard matrix manipulations which  
transform the problem of finding matrices $F$ which place the poles of the transfer 
function~\eqref{Eq:feedback} into a geometric problem on a Grassmannian. We then make some 
definitions. 
 
These poles are the roots of the characteristic polynomial $\varphi(s)=\det(sI_n-(A+BFC))$ 
of the matrix in~\eqref{Eq:feedback}. 
The rational function $\varphi(s)/\det(sI_n-A)$ equals the determinant of the product 
\[ 
  \left[\begin{matrix} 
      (sI_n-A)^{-1}&0&0\\-C(sI_n-A)^{-1}&I_m&0\\0&0&I_m\end{matrix}\right] 
  \left[\begin{matrix}sI_n-A-BFC&BF&-B\\0&I_m&0\\0&0&I_m\end{matrix}\right] 
  \left[\begin{matrix}I_n&0&0\\C&I_m&0\\0&F&I_m\end{matrix}\right] 
\] 
which is 
\[ 
  \left[\begin{matrix} 
      (sI_n-A)^{-1}&0&0\\-C(sI_n-A)^{-1}&I_m&0\\0&0&I_m\end{matrix}\right] 
  \left[\begin{matrix}sI_n-A&0&-B\\C&I_m&0\\0&F&I_m\end{matrix}\right] 
  \ =\  
  \left[\begin{matrix}I_n&0&-(sI_n-A)^{-1}B\\0&I_m&C(sI_n-A)^{-1}B\\0&F&I_m\end{matrix}\right].  
\] 
The transfer function $G(s)=C(sI_n-A)^{-1}B$ admits a left coprime factorization into 
matrices of polynomials $D(s)^{-1}N(s)$, where $\det D(s)=\det(sI_n-A)$, 
and so we have  
 \begin{equation}\label{Eq:geometry} 
   \varphi(s)\ =\ \det(sI_n-A)\det\left[\begin{matrix}I_m&G(s)\\F&I_m\end{matrix}\right] 
   \ =\ \det \left[\begin{matrix} D(s)&N(s)\\F&I_m\end{matrix}\right]. 
 \end{equation} 
Thus, $s$ is a pole of the transfer function of~\eqref{Eq:feedback} 
if and only if the matrix on the right of~\eqref{Eq:geometry} does not have full rank. 
Geometrically, if $K(s)$ is the row space of $[D(s)\co N(s)]$ and $H$ is the row space 
of $[F\co I_m]$, which are both $m$-planes in $\C^{2m}$, then $K(s)\cap H\neq\{0\}$. 
 
It follows that the feedback laws $F$ which place a given set of poles $s_1,\dotsc,s_n$ 
correspond to those $H\in OG(m)$ such that  
 \begin{equation}\label{Eq:SchubertVariety} 
   K(s_i)\cap H\ \neq\ \{0\}\,, \ \ \text{for each $i=1,\dotsc,n$.} 
 \end{equation} 
When $G(s)$ is skew-symmetric, then $n=2\ell$ and $\varphi(s)$ has only $\ell$ distinct 
roots, say  $s_1,\dotsc,s_\ell$. 
In this case, we also have that $K(s_i)$ is isotropic and the set of $H\in OG(m)$ 
satisfying~\eqref{Eq:SchubertVariety} defines a \demph{Schubert subvariety} of $OG(m)$, which 
  represents the first Chern class $c_1$ of the tautological bundle.  
A consequence of the surjectivity of the map $d\Psi_0$~\eqref{Eq:dPsi0} for generic $A,B$ is 
that these Schubert varieties meet generically transversally (in the open set 
consisting of $H$ of the form $[F\co I_m]$). 
Thus, when $\ell=\binom{m}{2}$ there are finitely many $H$ 
satisfying~\eqref{Eq:SchubertVariety}  for $i=1,\dotsc,\ell$, and their count is bounded 
above by the intersection number $\deg(c_1^\ell)$, which may be computed  
using the Schubert calculus on $OG(m)$~\cite{FuPr}. It is equal to
 \[ 
   \DeCo{d_m}\ :=\  
   \binom{m}{2}! \frac{1!\dotsb(m-2)!}{1! 3!\dotsb(2m-3)!}\,, 
 \] 
which is also the degree of $OG(m)$ in its natural embedding as the spinor variety. 
We complete the proof of Theorem~\ref{Th:three} by exhibiting a specific skew-symmetric 
transfer function $G(s)$ of McMillan degree $2\binom{m}{2}$ such that there are 
exactly $d_m$ skew-symmetric feedback laws placing any given $\binom{m}{2}$ 
real poles, and all the feedback laws are real. 
This also shows that generic systems have $d_m$ feedback laws.  
 
The argument uses a result of Purbhoo~\cite{Purbhoo} concerning the reality of the Wronski 
map, which we will transfer into the language of systems theory. 
Suppose that $\C^{2m}$ has ordered basis $\be_1,\dotsc,\be_{2m}$. 
Let \demph{$\gamma(s)$} be the vector-valued function $\gamma\colon\C\to\C^{2m}$ with 
components  
\[ 
   \bigl(\,1\,,\, s\,,\,  
    \frac{s^2}{2!}\,,\, \dotsc\,,\, \frac{s^{m-1}}{(m{-}1)!}\frac{1}{\sqrt{2}} 
    \ ,\  \frac{(-s)^{2m-2}}{(2m-2)!}\,,\,\dotsc,  
    \frac{(-s)^m}{m!}\,,\, \frac{(-s)^{m-1}}{(m{-}1)!}\frac{1}{\sqrt{2}}\,\bigr)\,. 
\] 
If we set $\DeCo{v_i(s)}:=\left(\frac{d}{ds}\right)^{i-1} \gamma(s)$ for $i=1,\dotsc,m{-}1$  
and $v_m(s):=\left(\frac{d}{ds}\right)^{m-1} \gamma(s)+ 
\frac{1}{\sqrt{2}}(e_m+(-1)^{m-1}e_{2m})$,  
then the row span of $v_1(s),\dotsc,v_m(s)$ is isotropic. 
While this defines a curve in $OG(m)$, it does not come from a skew-symmetric linear 
system, as it does not correspond to a strictly proper transfer function. 
However, the row span  \demph{$K(s)$} of the vectors  
$s^{2m-2}v_1(s^{-1}),\dotsc,s^{m-1}v_m(s^{-1})$, is still isotropic and it comes from a 
strictly proper skew-symmetric transfer function.
We display this for $m=5$, giving a $5\times 10$ matrix with rows
 $s^{2m-2}v_1(s^{-1}),\dotsc,s^{m-1}v_m(s^{-1})$: 
\[ 
   \left[\begin{matrix} 
    s^8&s^7&\frac{s^6}{2}&\frac{s^5}{3!}&\frac{1}{\sqrt{2}}\frac{s^4}{4!}& 
                        \rule{3pt}{0pt}\frac{1}{8!}& 
     -\frac{s}{7!}&\frac{s^2}{6!}&-\frac{s^3}{5!}&\frac{1}{\sqrt{2}}\frac{s^4}{4!}\\ 
    0&s^7&s^6&\frac{s^5}{2}&\frac{1}{\sqrt{2}}\frac{s^4}{3!}& 
        \rule{0pt}{17pt}\rule{3pt}{0pt} \frac{1}{7!}&  
     -\frac{s}{6!}&\frac{s^2}{5!}&-\frac{s^3}{4!}&\frac{1}{\sqrt{2}}\frac{s^4}{3!}\\ 
    0&0&s^6&s^5&\frac{1}{\sqrt{2}}\frac{s^4}{2}& 
       \rule{0pt}{17pt}\rule{3pt}{0pt} \frac{1}{6!}& 
     -\frac{s}{5!}&\frac{s^2}{4!}&-\frac{s^3}{3!}&\frac{1}{\sqrt{2}}\frac{s^4}{2}\\ 
    0&0&0&s^5&\frac{1}{\sqrt{2}}s^4& 
       \rule{0pt}{17pt}\rule{3pt}{0pt} \frac{1}{5!}& 
     -\frac{s}{4!}&\frac{s^2}{3!}&-\frac{s^3}{2!}&\frac{1}{\sqrt{2}}s^4\\ 
    0&0&0&0&\sqrt{2}s^4&\rule{0pt}{17pt}\rule{3pt}{0pt} \frac{1}{4!}& 
     -\frac{s}{3!}&\frac{s^2}{2!}&-s^3&0 
   \end{matrix}\right]. 
\] 
 
If we write $K(s)=[D(s)\co N(s)]$, then $D(s)$ is an upper triangular matrix with diagonal 
\mbox{$(s^{2m-2},\dotsc,s^m,\sqrt{2}s^{m-1})$,} and hence is 
invertible for all $s\neq 0$. Set $\DeCo{G(s)}:=D(s)^{-1}N(s)$, which is strictly proper and real. 
Here is $G(s)$ when $m=5$: 
\[ 
   \left[\begin{matrix} 
     0& \frac{5}{2}\frac{1}{7!s^7}&-\frac{5}{2}\frac{1}{6!s^6}& 
        \frac{3}{2}\frac{1}{5!s^5}&-\frac{1}{\sqrt{2}}\frac{1}{4!s^4}\\ 
    -\frac{5}{2}\frac{1}{7!s^7}&0\rule{0pt}{17pt}& \frac{1}{5!s^5}& 
      -\frac{1}{4!s^4}& \frac{1}{\sqrt{2}}\frac{1}{3!s^3}\\ 
    \frac{5}{2}\frac{1}{6!s^6}& -\frac{1}{5!s^5}&0\rule{0pt}{17pt}&  
      \frac{1}{2}\frac{1}{3!s^3}&-\frac{1}{\sqrt{2}}\frac{1}{2!s^2}\\ 
    -\frac{3}{2}\frac{1}{5!s^5}&\frac{1}{4!s^4}& 
      -\frac{1}{2}\frac{1}{3!s^3}&0\rule{0pt}{17pt}&\frac{1}{\sqrt{2}}\frac{1}{s}\\ 
    \frac{1}{\sqrt{2}}\frac{1}{4!s^4}& -\frac{1}{\sqrt{2}}\frac{1}{3!s^3}& 
   \frac{1}{\sqrt{2}}\frac{1}{2!s^2}&   -\frac{1}{\sqrt{2}}\frac{1}{s}&0\rule{0pt}{17pt} 
    \end{matrix}\right]. 
\]

Theorem~\ref{Th:three} follows from the following facts about the transfer function $G(s)$. 
 
\begin{proposition} 
 The transfer function $G(s)$ is skew-symmetric with McMillan degree $2\binom{m}{2}$. 
 Any set of $\binom{m}{2}$ distinct real poles is placed by  
 exactly $d_m$ skew-symmetric feedback laws, with each one real. 
\end{proposition} 
 
\begin{proof} 
  Since the isotropic $m$-plane $K(s)$ is the the row space of $[I_m\co G(s)]$, we conclude 
  that $G(s)$ is skew-symmetric.  
  Let $V\simeq \C^{2m-1}\subset\C^{2m}$  be the subspace with ordered basis 
\[ 
  (\be_1,\dotsc,\be_{m-1}, 
  (\be_m+(-1)^{m-1}\be_{2m})/\sqrt{2},\be_{m+1},\dotsc,\be_{2m-1})\,. 
\] 
  The nondegenerate symmetric bilinear form on $\C^{2m}$ 
  restricts to a nondegenerate symmetric bilinear form 
  on $V$, and the map $H\mapsto W:=H\cap V$ sends a maximal isotropic subspace $H$ of $\C^{2m}$ 
  to a maximal isotropic subspace of $V$, inducing an isomorphism between $OG(m)$ and 
  the space $BOG(m{-}1)$ of maximal isotropic subspaces of $V\simeq\C^{2m-1}$. 
  The reason for this is that for each $W\in BOG(m{-}1)$ there are two maximal isotropic subspaces 
  $H$ of $\C^{2m}$ containing $W$, exactly one of which lies in $OG(m)$. 
  When $W$ is real, both isotropic subspaces $H$ containing $W$ are also real. 
 
  Also, $\gamma(s)$ is a rational normal curve in $V$, as it involves the monomials 
  $1,\dotsc,s^{2m-1}$.  
  Furthermore, $\DeCo{L(s)}:=K(s)\cap V$ is the $(m{-}1)$-plane osculating 
  $\gamma(s^{-1})$, and $L(s)$ is isotropic.  
 
  The problem of which isotropic subspaces $W$ of $V$ that meet $r = \binom{m}{2}$ 
  osculating planes $L(s_1),\dotsc,L(s_{r})$ was studied by Purbhoo~\cite{Purbhoo} in the 
  context of the Wronski map from $BOG(m{-}1)\simeq OG(m)$, which extends the pole placement map 
  from $[F\co I_m]\mapsto\varphi(s)$ for the transfer function $G(s)$. 
  This map is surjective onto the space of polynomials of degree $2r$ which are 
  squares of polynomials, and it has finite fibers of algebraic degree $d_m$. 
  This implies that there are at most $d_m$ feedback laws placing a given set of 
  $r$ poles. 
  It also implies that any given isotropic plane $H$ meets at most $r$ 
  subspaces of the form $L(s)$, including $L(\infty)=[0\co I_{m-1}]$. 
 
  Purbhoo~\cite[Theorem 3]{Purbhoo} showed that if  
  $s_1,\dotsc,s_{r}$ were real, then there are 
  exactly $d_m$ real isotropic planes $W$ in $BOG(m-1)$ such that  
  $L(s_i)\cap W\neq\{0\}$, for each $i=1,\dotsc,r$. 
  For each such $W$, let $H$ be the unique isotropic plane in $OG(m)$ containing $W$, which 
  is necessarily real. 
  Then $H\cap K(s_i)\neq 0$ for each $i$, and so $H$ corresponds to a real feedback law  
  if $H$ has the form $[I_m\co F]$. 
  But this is guaranteed for otherwise $H\cap K(\infty)\neq 0$, which would imply  
  $W\cap L(\infty)\neq\{0\}$, an impossibility as $H$ already meets the maximum number of 
  subspaces of the form $L(s)$. 
  Lastly, the transfer function has McMillan degree $2r$ since the image of the pole 
  placement map (a linear projection) meets the set of polynomials of this degree. 
\end{proof} 
 
\providecommand{\bysame}{\leavevmode\hbox to3em{\hrulefill}\thinspace} 
\providecommand{\MR}{\relax\ifhmode\unskip\space\fi MR } 
\providecommand{\MRhref}[2]{%
  \href{http://www.ams.org/mathscinet-getitem?mr=#1}{#2} 
} 
\providecommand{\href}[2]{#2}

\end{document}